\documentclass[11pt,a4paper,reqno]{amsart}

\usepackage{titletoc}
\usepackage{titlesec}
\titleformat{\section}[block]{\Large\bfseries}{\thesection}{1em}{}
\titleformat{\subsection}[block]{\bfseries}{\thesubsection}{1em}{}
\dottedcontents{section}[1.5cm]{}{3em}{1pc}
\dottedcontents{subsection}[1.5cm]{}{3em}{1pc}

\usepackage{amsmath, amsthm, amssymb}
\usepackage{amsfonts}
\usepackage[utf8]{inputenc}
\usepackage[dvips]{epsfig}
\usepackage{graphicx}
\usepackage[english]{babel}
\usepackage{hyperref}
\hypersetup{colorlinks=true,linkcolor=black,citecolor=black}

\usepackage{mathptmx}
\usepackage{cite}
\usepackage{graphicx}

\allowdisplaybreaks
\usepackage{color}
\usepackage[dvipsnames]{xcolor}

\def\sideremark#1{\ifvmode\leavevmode\fi\vadjust{\vbox to0pt{\vss
 \hbox to 0pt{\hskip\hsize\hskip1em
 \vbox{\hsize2.1cm\tiny\raggedright\pretolerance10000
  \noindent #1\hfill}\hss}\vbox to15pt{\vfil}\vss}}}%

\let\oldsqrt\sqrt
\def\sqrt{\mathpalette\DHLhksqrt}
\def\DHLhksqrt#1#2{%
\setbox0=\hbox{$#1\oldsqrt{#2\,}$}\dimen0=\ht0
\advance\dimen0-0.2\ht0
\setbox2=\hbox{\vrule height\ht0 depth -\dimen0}%
{\box0\lower0.4pt\box2}}

\newcommand{\R}{\mathbb{R}} 
\newcommand{\N}{\mathbb{N}} 

\newcommand{\sOmega}{\text{\tiny $\Omega$}}
\newcommand{\cEO}{{{E}_{\sOmega}\,}}

\newcommand{\loglap}{L_{\text{\tiny $\Delta \,$}}}

\newcommand{\dist}{\textnormal{dist}} 
\newcommand{\diam}{\textnormal{diam}} 

\renewcommand{\phi}{\varphi}

\newcommand{\cH}{{\mathcal H}}

\newcommand{\cL}{{\mathcal L}}

\newcommand{\cN}{{\mathcal N}}

\newcommand{\cGs}{{\mathbb G}}

\newcommand{\cPs}{{\mathbb P}}
\newcommand{\cLs}{{\mathbb L}}

\newcommand{\eps}{\varepsilon}

\theoremstyle{definition}
\newtheorem{defi}{Definition}[section]
\newtheorem{remark}[defi]{Remark}

\theoremstyle{plain} 
\newtheorem{thm}[defi]{Theorem}
\newtheorem{prop}[defi]{Proposition}
\newtheorem{lemma}[defi]{Lemma}
\newtheorem{cor}[defi]{Corollary}

\theoremstyle{definition}

\numberwithin{equation}{section}

\title[Differentiability of the nonlocal-to-local transition in fractional Poisson problems]{Differentiability of the nonlocal-to-local transition in fractional Poisson problems}

\author[Sven Jarohs]
{Sven Jarohs}
\address{Institut f\"ur Mathematik,
Goethe-Universit\"at Frankfurt.
Robert-Mayer-Str. 10
D-60629 Frankfurt am Main, Germany}
\email{jarohs@math.uni-frankfurt.de}

\author[Alberto Salda\~{n}a]
{Alberto Salda\~{n}a}
\address{Instituto de Matemáticas\\
Universidad Nacional Autónoma de México\\
Circuito Exterior, Ciudad Universitaria\\
04510 Coyoacán, Ciudad de México, Mexico\\
}
\email{alberto.saldana@im.unam.mx}

\author[Tobias Weth]
{Tobias Weth}
\address{Institut f\"ur Mathematik,
Goethe-Universit\"at Frankfurt.
Robert-Mayer-Str. 10
D-60629 Frankfurt am Main, Germany}
\email{weth@math.uni-frankfurt.de}

\date{}

\begin{document}

\begin{abstract}
Let $u_s$ denote a solution of the fractional Poisson problem
\begin{align*}
    (-\Delta)^s u_s = f\quad\text{ in }\Omega,\qquad u_s=0\quad \text{ on }\R^N\setminus \Omega,
\end{align*}
where $N\geq 2$ and $\Omega\subset \R^N$ is a bounded domain of class $C^2.$ We show that the solution mapping $s\mapsto u_s$ is differentiable in $L^\infty(\Omega)$ at $s=1$, namely, at the nonlocal-to-local transition. Moreover, using the logarithmic Laplacian, we characterize the derivative $\partial_s u_s$ as the solution to a boundary value problem.  This complements the previously known differentiability results for $s$ in the open interval $(0,1)$.  Our proofs are based on an asymptotic analysis to describe the collapse of the nonlocality of the fractional Laplacian as $s$ approaches 1.
We also provide a new representation of $\partial_s u_s$ for $s \in (0,1)$ which allows us to refine previously obtained Green function estimates.
\medskip

\medskip

\end{abstract}

\maketitle

\section{Introduction}

Let $N\geq 2,$ $\Omega\subset \R^N$ be a bounded $C^{2}$-domain, and let $f\in C^\alpha(\overline{\Omega})$ for some $\alpha\in(0,1)$.
The main purpose of this paper is to study the nonlocal-to-local transition of the family of fractional Dirichlet-Poisson problems
\begin{align*}
(D_s)\hspace{2cm}     (-\Delta)^s u = f\quad\text{ in }\Omega,\qquad u=0\quad \text{ on }\R^N\setminus \Omega,
\end{align*}
to the classical counterpart
\begin{align*}
(D_1)\hspace{2.8cm}    -\Delta u = f\quad\text{ in }\Omega,\qquad u=0\quad \text{ on }\partial \Omega,
\end{align*}
arising in the limit $s \to 1^-$. Here, $(-\Delta)^s$ is the fractional Laplacian corresponding to the Fourier symbol $|\cdot|^{2s}$.  In particular, for $s\in(0,1)$, $\phi \in C^{\infty}_c(\R^N)$, and $x\in\R^N$,
\begin{align*}
 (-\Delta)^s\phi(x)=c_{N,s}\lim_{\epsilon\to0^+}\int_{\R^N\setminus B_{\epsilon}(x)}\frac{\phi(x)-\phi(y)}{|x-y|^{N+2s}}\ dy,\qquad
    c_{N,s}=\frac{4^{s}\Gamma(\frac{N}{2}+s)}{\Gamma(2-s)\pi^{\frac{N}{2}}}s(1-s).
\end{align*}
 It is well known (see, for example, \cite[Lemma 3.4]{JSW20}) that $(D_s)$ admits a unique solution $u_s\in C^s(\R^N)\cap C^{2s+\epsilon}_{loc}(\Omega)$ for $\eps \in (0,\alpha)$, while $(D_1)$ admits a unique solution $u_1 \in C^1(\overline \Omega)\cap C^{2+\alpha}_{loc}(\Omega)$. For matters of consistency, we identify $u_1$ with its trivial extension outside of $\Omega$.
In \cite{JSW20}, it is shown that the mapping $s\mapsto u_s$ is of class $C^1((0,1),L^\infty(\Omega))$ and its derivative $v_s=\partial_s u_s \in L^\infty(\Omega)$ can be characterized in terms of a fractional problem involving the logarithmic Laplacian.  To be more precise, if $s\in(0,1)$, then $v_s$ is the unique weak solution of
\begin{equation}\label{char-derivative-old-simple}
    (-\Delta)^s v_s=-\loglap [(-\Delta)^s u_s] \quad \text{in }\Omega,\qquad v_s=0\quad \text{ on }\R^N\setminus\Omega,
\end{equation}
where $(-\Delta)^s u_s$ is to be evaluated on all of $\R^N$ and $\loglap$ denotes the logarithmic Laplacian introduced in\cite{CW19}. For a function $\phi \in L^1(\R^N;(1+|y|)^{-N}dy)$ which is Dini continuous at a point $x \in \R^N$, the logarithmic Laplacian of $\phi$ at $x$ is given by
\begin{align}\label{omegarep-0}
\loglap \phi(x)=c_{N}\int_{\R^N}\frac{\phi(x)1_{B_1(x)}(y)-\phi(y)}{|x-y|^N}\,dy+\rho_N \phi(x),
\end{align}
where
\begin{align*}
c_N:=\pi^{-\tfrac{N}{2}}\Gamma(\tfrac{N}{2}),\qquad \rho_{N}:=2\ln 2+\psi(\tfrac{N}{2})-\gamma,\qquad \gamma:=-\Gamma^{\prime}(1),
\end{align*}
and $\psi=\frac{\Gamma^{\prime}}{\Gamma}$ is the digamma function. Moreover, if $\phi \equiv 0$ on $\R^N\setminus \Omega$, we have the representation\begin{align}\label{omegarep}
\loglap \phi(x)=c_{N}\int_{\Omega}\frac{\phi(x)-\phi(y)}{|x-y|^N}dy + (h_{\Omega}(x)+\rho_N)\phi(x),\qquad x\in \R^N,
\end{align}
with the function
\begin{align*}
x \mapsto h_{\Omega}(x):=c_N\int_{B_1(x)\setminus \Omega} |x-y|^{-N}\ dy-c_N\int_{\Omega\setminus B_1(x)}|x-y|^{-N}\ dy,
\end{align*}
see \cite[Proposition 2.2]{CW19}.  As shown in \cite[Theorem 1.1]{CW19}, it follows that $\loglap$ has the Fourier symbol $2\ln|\cdot|$ and it can be seen as the derivative in $s$ of $(-\Delta)^s$ at $s=0$.

The formula (\ref{char-derivative-old-simple}) is short but unsatisfactory because the dependence of the source function $f$ does not appear explicitly on the RHS. It is therefore useful to write
$(-\Delta)^s u_s = E_\Omega f + w_s$, where $E_\Omega f$ denotes the trivial extension of $f$ to $\R^N$
and
\begin{align*}
w_s= [(-\Delta)^s u_s]1_{\R^N\setminus \Omega}, \qquad \text{i.e.,}\qquad w_s(x)=- c_{N,s}1_{\R^N\setminus \Omega}(x)\int_{\Omega}\frac{u_s(y)}{|x-y|^{N+2s}}\, dy.
\end{align*}
Then (\ref{char-derivative-old-simple}) becomes
\begin{equation}\label{char-derivative-old}
    (-\Delta)^s v_s=-\loglap \bigl(E_\Omega f + w_s \bigr),\qquad v_s=0\quad \text{ on }\R^N\setminus\Omega.
\end{equation}
Here, the dependency of $w_s$ on $f$ is still implicit, but nevertheless the characterization \eqref{char-derivative-old} of $v_s$ is useful to describe, under some assumptions on $f$ and $\Omega$, qualitative properties of the solution mapping $s\mapsto u_s$ such as monotonicity and $L^\infty$-bounds that take into account the geometry of the domain, see \cite{JSW20}.

If $f \ge 0$, $f \not\equiv 0$ in $\Omega$, then the strong maximum principle implies $u_s >0$ in $\Omega$ and therefore $w_s<0$ in $\R^N\setminus \Omega$. Since $w_s= 0$ in $\Omega$, it follows that $\loglap w_s > 0$ in $\Omega$. If, in addition, $\loglap E_\Omega f \ge 0$ in $\Omega$,
it then follows from (\ref{char-derivative-old}) and the strong maximum principle that $v_s < 0$ in $\Omega$ for $s \in (0,1)$, so the map $s \mapsto u_s$ is pointwisely strictly decreasing in $s$ on $\Omega$.

In \cite{JSW20}, the differentiability of $s\mapsto u_s$ is only shown for $s\in(0,1)$. In fact, many of the estimates involved in the proofs in \cite{JSW20} are not robust in the limit as $s\to 1^-$, namely, one obtains degenerating constants and limiting non-integrable functions.  For example, in \eqref{char-derivative-old}, the term $w_s$ behaves like $(1-s)\dist(x,\partial\Omega)^{-s}$ close to $\partial \Omega$.

It is well known that, when $s\to 1^-$, the unique solution $u_s$ of $(D_s)$ converges to $u_1 \in  C^1(\overline{\Omega})\cap C^{2+\epsilon}_{loc}(\Omega)$, the unique solution of $(D_1)$. This is sometimes referred to as the \emph{nonlocal-to-local transition}, and this phenomenon has been studied for linear and nonlinear problems in unbounded and bounded domains with Dirichlet and Neumann boundary conditions, see for example \cite{HS21,BH18,GH22,BS202,BS21,BS20,GKV20,GK23,G23}.\\

In this paper, we refine the analysis in \cite{JSW20} and show differentiability of the map $s \mapsto u_s$ up to $s = 1$. To state our main result, we need to recall some more notation. For $s \in (0,1)$, we let $\cH^s_0(\Omega):=\{u\in H^s(\R^N)\::\: u=0\text{ in }\R^N\setminus \Omega\}$, and for consistency we put
$\cH^1_0(\Omega):=H^1_0(\Omega)$, which, as usual, is defined as the completion of $C^\infty_c(\Omega)$-functions with respect to the $H^1$-norm.

 A key role in our analysis is played by the Poisson kernel $P_1: \Omega \times \partial \Omega \to \R$ for the Laplacian and the fractional Poisson kernel
$P_s: \Omega \times \Omega^c\to\R$ associated to $(-\Delta)^s$ in $\Omega$. We shall recall the main properties of $P_s$ and $P_1$ in Section \ref{Sec:est} below. The classical Poisson kernel is simply defined as
\begin{equation}
  \label{eq:def-classical-poisson}
P_1: \Omega \times \partial \Omega \to \R, \qquad P_1(x,z) = -\partial_{\nu_z}G_1(x,z),
\end{equation}
where $G_1$ is the Green function for $-\Delta$ on $\Omega$ and $\nu_z$ is the outward normal of $\partial\Omega$ at $z$, and it has the property that, for every continuous function $g$ on $\partial \Omega$,
the function $x \mapsto  \int_{\partial \Omega}P_1(x,z)g(z)\,\sigma(dz)$
is the harmonic extension of $g$ on $\Omega$, where $\sigma$ denotes the surface measure of $\partial \Omega$.

For $s \in (0,1)$, the Poisson kernel $P_s: \Omega \times (\R^N\setminus\Omega) \to \R$ is defined as
\begin{align}\label{Psdef}
P_s(x,z):=c_{N,s}\int_\Omega \frac{G_s(x,y)}{|y-z|^{N+2s}}\, dy\qquad \text{for $x\in \Omega$, $z\in \R^N\setminus\Omega$,}
\end{align}
where $G_s$ is the Green function for $(-\Delta)^s$ on $\Omega$. It has the property that, for every $g\in H^1(\R^N\setminus\Omega)\cap C^{1,\alpha}(\R^N\setminus\Omega)$, the function $x \mapsto \int_{\R^N\setminus\Omega}P_s(x,z)g(z)dz$
is the unique $s$-harmonic extension of $g$ on $\Omega$ which is bounded in a neighborhood of $\partial \Omega$ (see e.g. \cite{BGP20}).

We emphasize that the connection between \eqref{eq:def-classical-poisson} and \eqref{Psdef} follows through the {\em nonlocal normal derivative} introduced in \cite{DRV17}, which, for functions $v: \R^N \to \R$ with $v \big |_{\Omega} \in L^1(\Omega)$, is defined by 
	$$
	\cN_sv(z)=c_{N,s}\int_{\Omega}\frac{v(z)-v(y)}{|z-y|^{N+2s}}\,dy\quad\text{for $z\in \R^N \setminus \overline{\Omega}$.}
	$$
With this notation and since $G(x,y)=0$ for $x\in \Omega$, $y\in \R^N\setminus \Omega$, \eqref{Psdef} can also be stated as
$$
P_s(x,z)=-[\cN_{s}G_s(x,\cdot)](z)  \quad\text{for $x\in \Omega$, $z\in \R^N\setminus \overline{\Omega}$.}
$$
The connection between \eqref{eq:def-classical-poisson} and \eqref{Psdef} then can be seen through \cite[Proposition 5.1]{DRV17}, which states that for $v,w\in C^2_c(\R^N)$ we have
$$
\lim_{s\to1}\int_{\R^N\setminus \Omega}\cN_sv(x)w(x)\,dx=\int_{\partial\Omega}\partial_{\nu_y}v(y)w(y)\,d \sigma(y).
$$

While these Poisson kernels are well-known, it will be convenient to introduce now the new notion of {\em complementary Poisson kernels} $P_s^{\,c}$, $s \in (0,1]$.

\begin{defi}
\label{complementary-poisson-kernel}
  For $s \in (0,1]$, we define the complementary Poisson kernel
  $$
  P_s^{\,c}: \Omega \times \Omega \to \R
  $$
  of order $s$ by
  \begin{equation}
    \label{eq:def-P-s-c}
    P_s^{\,c}(x,z) = c_N \int_{\R^N\setminus \Omega}\frac{ P_s(z,y)}{|x-y|^{N}}\,dy \quad \text{for $s\in (0,1)$}\quad \text{and}\quad
P_1^{\,c}(x,z) = c_N \int_{\partial \Omega}\frac{ P_1(z,y)}{|x-y|^{N}}\,d \sigma(y).
  \end{equation}
\end{defi}

We may now state our first main result.

\begin{thm}\label{thm1:differential}
  Let $N\geq 2$, let $\Omega\subset\R^N$ be an open and bounded set of class $C^2$, and let $f\in C^{\alpha}(\overline{\Omega})$ for some $\alpha>0$. Moreover, for $s \in (0,1]$, let $u_s \in L^\infty(\Omega)\cap \cH^s_0(\Omega)$ be the unique solution of $(D_s)$. Then the map $(0,1] \to L^\infty(\Omega)$, $s \mapsto u_s$ is of class $C^1$. Moreover, for $s \in (0,1]$, the derivative $v_s:= \partial_s u_s \in L^\infty(\Omega)\cap \cH^s_0(\Omega)$ is the unique weak solution of the problem
  \begin{equation}
    \label{char-derivative-new-1}
    (-\Delta)^s v_{s} = \cLs_s f  \quad \text{in $\Omega$,}\qquad v_s=0\quad\text{in $\R^N\setminus \Omega$}
  \end{equation}
  with
  \begin{equation}
    \label{L-s-representation}
    \cLs_s f(x) = -\loglap E_\Omega f(x) - \int_{\Omega} P_s^{\,c}(x,z)f(z)\,dz,
  \end{equation}
  where $E_\Omega f$ denotes the trivial extension of $f$ on $\R^N$. In addition, if $f \geq 0$ in $\Omega$ and $f\not\equiv0$, then
\begin{align*}
v_s < 0\quad \text{ in $\Omega\quad $ for all $s \in (0,1]$}\qquad \text{ if and only if }\qquad \loglap \cEO f \ge 0\quad \text{in $\Omega$.}
\end{align*}
\end{thm}

\begin{remark}
\label{remark-thm1:differential}
(i) The function $\cLs_s f: \Omega \to \R$ on the RHS of \eqref{char-derivative-new-1}
can be represented shortly as
$$
\cLs_s f = -(-\Delta)^s \loglap u_s\qquad \text{in $\Omega$ for $s \in (0,1]$,}
$$
where again we assume that $u_1$ is trivially extended to all of $\R^N$.\\
(ii) Observe that $v_1$ is a solution of a local problem involving the Laplacian and that the right-hand side can be singular at the boundary $\partial \Omega$.  As a consequence, the behavior of $v_1$ may be slightly worse than Lipschitz at $\partial\Omega$, see Remark \ref{rmk:boundary}.\\
(iii) The solution map is also differentiable at $s=0$, but not in the $L^\infty(\Omega)$ sense. Instead, one must use a weighted space $L^\infty(\Omega;\dist(\cdot,\partial \Omega)^a)$ for $a>0$, see \cite[Proposition 4.1]{JSW20}.  Furthermore, we point out that Theorem \ref{thm1:differential} gives an alternative expression for \eqref{char-derivative-old} in the case $s \in (0,1)$.\\
(iv) The complementary Poisson kernel $P_s^{\,c}$ is a natural object that appears in our calculations to keep track of the explicit dependence on $f$ of the derivative $v_s=\partial_s u_s$. It also helps to describe the nonlocal-to-local transition; indeed, in Corollary \ref{pointwise:lem} we show that $P_s^{\,c}(x,\cdot)\to P_1^{\,c}(x,\cdot)$ in $L^2(\Omega)$ as $s\to 1^-$.  Furthermore, these kernels are also helpful in improving some bounds on the operator norm of the Green function $\cGs_s$, see Theorem~\ref{new-theorem-bound-green} below. \\
(v) By definition, the function $v_1$ is given as a one sided lower derivative
$$
v_1 := \lim_{t \to 0^+}\frac{u_1-u_{1-t}}{t}.
$$
If the domain is smooth, we also show that the derivative from above coincides with the limit from below. This is the content of the next theorem.
\end{remark}

\begin{thm}\label{thm1:differential:2}
Let $N\geq 2$, let $\Omega\subset\R^N$ be an open bounded set of class $C^\infty$ and let $f\in C^{\alpha}(\overline{\Omega})$ for some $\alpha>0$. Then
\begin{align*}
v_{1^-}:=\lim_{t\to 0^+}\frac{u_1-u_{1-t}}{t}=\lim_{t\to 0^+}\frac{u_1-u_{1+t}}{t}=:v_{1^+}\qquad\text{in $L^{\infty}(\Omega)$.}
\end{align*}
\end{thm}
The extra smoothness assumption on the domain is only to guarantee the regularity of solutions in the higher-order regime and can probably be weakened, but we do not pursue this here.

As a Corollary, we obtain the following asymptotic expansion of the solution mapping for $s$ close to 1.

\begin{cor}\label{cor:expansion at 1}
Under the assumptions of Theorem \ref{thm1:differential}, we have that
\begin{align*}
u_s=u_1+(1-s)v_1+o(1-s)\qquad \text{ in }L^\infty(\Omega)\text{ as }s\to 1^-,
\end{align*}
uniformly on $\Omega$, where $v_1 \in \cH^1_0(\Omega)$ is the unique weak solution of
$$
-\Delta v_1 = \cLs_1 f \quad \text{in $\Omega$,}\qquad v_1 = 0\quad \text{on $\partial \Omega$}
$$
with $\cLs_1 f$ given by (\ref{L-s-representation}) with $s=1$.
\end{cor}

To show Theorems \ref{thm1:differential} and \ref{thm1:differential:2}, we perform a careful asymptotic analysis of the limit from below $\lim_{s\to 1^-}v_s$ and of the limit from above $\lim_{s\to 1^+}v_s$, using the results in \cite{CW19,JSW20}. Note that the limit from above requires to consider the higher-order regime where $s>1$, which is much less studied in the literature; we refer to \cite{AJS18,G15,S21,A22} and the references therein for an overview of these operators and for some explicit representation formulas in this setting.

The limit from below is delicate, since we require robust estimates that describe the \textquotedblleft collapse" of the nonlocality. To illustrate this point, we fix $g\in H^1(\Omega^c)\cap C^{1,\alpha}(\Omega^c)$, and we let, for $s \in (0,1)$, the function $h_s$ denote the $s$-harmonic extension of $g$, namely the solution of
\begin{align}\label{nhDs}
    (-\Delta)^{s} h_s = 0\quad \text{ in }\Omega,\qquad h_s = g\quad \text{ on }\R^N\backslash \Omega,
\end{align}
which in terms of the Poisson kernel writes as
\begin{align}\label{Psg}
x\mapsto h_s(x):= \int_{\R^N\backslash\Omega}P_{s}(x,y)g(y)\,dy\:  1_{\Omega}(x) + g1_{\R^N\backslash \Omega}(x)
\in C^\infty(\Omega)\cap C^{s}(\R^N),
\end{align}
As $s \to 1^-$, $h_s$ converges in $L^2(\Omega)$ to the solution $h_1 \in C^\infty(\Omega)\cap C^1(\overline{\Omega})$ of
\begin{align}\label{nhD}
    -\Delta \,h_1 = 0\quad \text{ in }\Omega,\qquad h_1=g\quad \text{ on }\partial \Omega,
\end{align}
see \cite[Theorem 5.80]{G20} or Theorem \ref{conv:thm} below. Moreover, we have
\begin{align}\label{P1g}
h_1 := \int_{\partial\Omega}P_1(\cdot,y)g(y)\,d \sigma(y),
\end{align}
so the integral over $\R^N\setminus \Omega$ in \eqref{Psg} is reduced to an integral on the boundary $\partial \Omega$ in \eqref{P1g}.  This is mainly due to the fact that the measure $m\,dx$ with
\begin{align*}
    m=\frac{1-s}{\dist(\cdot,\partial \Omega)^s (1+\dist(\cdot,\partial \Omega)^{N+s})}
\end{align*}
converges weakly to the surface measure on $\partial \Omega$, see \cite[Lemma 4.1]{GH22}. In Proposition \ref{prop:below} the limit from below $v_{1^-}$ is calculated and in Proposition \ref{prop:above} we determine the limit from above $v_{1^+}$. Then, to show that these limits coincide, we use a formula for the interchange of a Laplacian with a logarithmic Laplacian, as stated in the following result of independent interest.

\begin{lemma}\label{lem:interchange}
  For $s\in(0,1]$, let $u: \R^N \to \R$ be Lipschitz continuous if $s=1$ and $u\in C^{s}(\R^N)$ if $s\in(0,1)$. Moreover, suppose that $u \in C^{2s+\alpha}_{loc}(\Omega)$ for some $\alpha>0$, $u=0$ on $\R^N\backslash \Omega$, and that $[(-\Delta)^su]\big|_{\Omega}$ extends to a function in $C^{\alpha}(\overline{\Omega})$. Then, for all $x\in \Omega$,
  \begin{equation}
  \label{eq:lemma-interchange}
(-\Delta)^s L_{\Delta} u(x)=\left\{\begin{aligned}
& L_{\Delta}[-\Delta u](x)
+\int_{\Omega} P_1^{\,c}(x,z) (-\Delta) u (z)\,dz&&\qquad  \text{for }s=1;\\
&L_{\Delta}[(-\Delta)^s u](x)
&&\qquad \text{for }s<1.
\end{aligned}\right.
\end{equation}
Here, on the RHS of (\ref{eq:lemma-interchange}), the functions $-\Delta u: \R^N \to \R$ resp. $(-\Delta)^s u: \R^N \to \R$ are understood as the regular parts of the distributions $-\Delta u$, $(-\Delta)^s u$, respectively.
\end{lemma}

The different formulas on the RHS of (\ref{eq:lemma-interchange})  can be seen as a purely nonlocal-to-local phenomenon. We emphasize that for $s<1$ the operators commute, as the distribution $(-\Delta)^s u \in L^1(\R^N)$ is regular and pointwisely well-defined on $\R^N \setminus \partial \Omega$. On the contrary, for $s=1$, the distribution $-\Delta u$ has a singular part supported on $\partial \Omega$, and this singular part gives rise to the extra term on the RHS of (\ref{eq:lemma-interchange}). In other words, the counterpart of this extra term in the case $s<1$ is completely absorbed into the expression $L_{\Delta}[(-\Delta)^s u]$. See Remark~\ref{absorbed-term} below for a more detailed explanation of this aspect.

We also mention the reference \cite[Theorem~5.3]{BG20}, where the $C^\infty$-differentiability of the solution mapping $s\mapsto u_s$ from  the open interval $(0,1)$ into $L^2(\Omega)$ is obtained for smooth domains by a different method.
These differences can be seen as a purely nonlocal-to-local phenomenon.
To close this paper, we emphasize that a characterization as the one given in Theorem \ref{thm1:differential} and in \cite[Theorem 1.1]{JSW20} is still an open problem for any $s>1.$

\medskip

The paper is organized as follows. In Section \ref{Sec:est} we recall some known results about the convergence of solutions to fractional inhomogeneous Dirichlet problems as $s\to 1^-$ and apply them to some particular cases needed to show our main results.  The proof of our main theorems is contained in Section \ref{char:sec}. In Section~\ref{sec:bounds-green-oper}, we then derive Green operator bounds by applying Theorem~\ref{thm1:differential} to the special case of the fractional torsion problem. These bounds are contained in Theorem~\ref{new-theorem-bound-green} below. Finally, Lemma \ref{lem:interchange} is shown in the appendix, where we also included some uniform regularity estimates needed in our proofs.

\section{Convergence of solutions to fractional Dirichlet problems in the nonlocal-to-local limit}\label{Sec:est}

In this section, we provide some preliminary convergence results, as $s \to 1^-$, related to homogeneous and nonhomogeneous fractional Dirichlet problems involving the operator $(-\Delta)^s$. We start by fixing some notation. Let
$$
\R^{2N}_* := \R^N \times \R^N \setminus \{(x,x)\:: \: x \in \R^N\},
$$
and let $\Omega$ be a $C^2$-domain in $\R^N$. For $s\in(0,2)$, let $G_s: \R^{2N}_* \to \R^N$ denote the Green kernel associated to $(-\Delta)^s$ in $\Omega$. For $s\in(0,1]$, this kernel is by now well-known ---particularly, for $s\in(0,1)$, see e.g. \cite[Section 3]{JSW20} and the references in there--- and, for $s>1$, see \cite[Appendix A]{AJS18}.

For $r \in (1,\frac{N}{2s})$ we can define the Green operator
\begin{equation}
  \label{eq:G-op-definition}
\cGs_s \::\: L^{r}(\Omega) \to L^{p(r,s)}(\R^N)\quad \text{ given by }\quad [\cGs_s f](x) = \int_{\Omega}G_s(x,y)f(y)\,dy \qquad \text{for $x \in \R^N$,}
\end{equation}
where $p(r,s)=\frac{rN}{N-2sr}$. We emphasize that $\cGs_s$ can also be defined on $L^r(\Omega)$ for any $r\in[1,\infty]$, but we do not need this in the following.
Moreover, we define, for $s \in (0,1)$, the associated Poisson operators
\begin{align*}
\cPs_s \:&:\: L^{\infty}(\R^N\setminus \Omega)\cap C^\alpha (\R^N\setminus \Omega) \to C^\infty(\Omega)\cap L^\infty(\R^N),\\
[\cPs_s h](x) &= \Bigl(\int_{\R^N\setminus \Omega}P_s(x,y)h(y)\,dy\Bigr)  1_{\Omega}(x) + h(x)1_{\R^N\setminus \Omega}(x) \qquad \text{for $x \in \R^N$,}
\end{align*}
where the (fractional) Poisson kernel $P_s$ is given in (\ref{Psdef}). In the case $s=1$, we define the corresponding Poisson operator by
\begin{align*}
\cPs_1 \:&:\: L^{\infty}(\R^N\setminus \Omega)\cap C^\alpha (\R^N\setminus \Omega) \to C^\infty(\Omega)\cap L^\infty(\R^N),\\
[\cPs_1 h](x) &= \Bigl(\int_{\partial\Omega}P_1(x,y)h(y)\,d \sigma(y)\Bigr)1_{\Omega}(x) + h(x)\,1_{\R^N\setminus \Omega}(x) \qquad \text{for $x \in \R^N$},
\end{align*}
where $P_1$ is given in (\ref{eq:def-classical-poisson}) and $\sigma$ denotes the surface measure of $\partial \Omega$.

If $f$ is a function defined on $\Omega$, then $E_\Omega f$ denotes the trivial extension of $f$ to $\R^N$. For $\Omega\subset \R^N$ we also write $\Omega^c:=\R^N\setminus \Omega$ for the complement of $\Omega$ in $\R^N$.

To characterize the derivative of the solution map in the nonlocal-to-local transition, we need first to recall some results about the convergence of solutions to fractional inhomogeneous Dirichlet problems as $s\to 1^-$ and apply them in some particular cases of interest.

First, we recall the following result stating the convergence of inhomogeneous Dirichlet fractional problems.

\begin{thm}\label{conv:thm}(see \cite[Theorem 5.80]{G20})
Let $g\in H^1(\Omega^c)\cap C^{1,\alpha}(\Omega^c)$. Moreover, let, for $s \in (0,1)$, the function $h_s$ be as in \eqref{nhDs}, and let $h_1$ be defined as in \eqref{nhD}. Then $h_s\to h_1$ in $L^2(\Omega)$ as $s \to 1^-$.
\end{thm}

We note here that \cite[Theorem 5.80]{G20} assumes complementary data $g\in H^1(\R^N)$. However, since $\Omega$ is of class $C^2$, one can assume that $g\in H^1(\Omega^c)$ and use a suitable extension. Hence, Theorem~\ref{conv:thm} is a direct consequence of \cite[Theorem 5.80]{G20}.

As a corollary of Theorem~\ref{conv:thm}, we get the following $L^2$-convergence result for the complementary Poisson kernels defined in Definition \ref{complementary-poisson-kernel}.

\begin{cor}\label{pointwise:lem}
  For every $x \in \Omega$ we have $\|P_s^{\,c}(x,\cdot)-P_1^{\,c}(x,\cdot)\|_{L^2(\Omega)}  \to 0$ as $s \to 1^-$.
\end{cor}
\begin{proof}
For $x\in\Omega$, we may define $g_x\in C^{1,\alpha}(\Omega^c )$ by $g_x(y):=|x-y|^{-N}$ for $y\in \Omega^c$. Moreover, with $r:=\dist(x,\partial \Omega)$, we have
\begin{align*}
    \int_{\Omega^c}|g_x(y)|^2\, dy
    \leq \int_{\R^N\setminus B_r(0)}|y|^{-2N}\, dy
    =|\partial B_1| \int_r^\infty \rho^{-2N+N-1} \, d\rho<\infty
\end{align*}
since $N\geq 2$ and, similarly,
\begin{align*}
    \int_{\Omega^c}|\nabla g_x(y)|^2\, dy
    \leq N^2\int_{\R^N\backslash B_r(0)}|y|^{-2(N+1)}\, dy
    =N^2|\partial B_1| \int_r^\infty \rho^{-2N+N-3} \, d\rho<\infty.
\end{align*}
Thus $g_x\in H^1(\Omega^c )\cap C^{1,\alpha}(\Omega^c)$ and therefore
\begin{align*}
    h_s:=\int_{\Omega^c}P_{s}(\cdot ,y)g_x(y)\,dy\to \int_{\partial \Omega}P_{1}(\cdot ,y)g_x(y)\,d \sigma(y)=:h_1\qquad \text{ in }L^2(\Omega)\ \text{ as }n\to\infty.
\end{align*}
by Theorem \ref{conv:thm}. The claim now follows by noting that $c_N h_s = P_s^{\,c}(x,\cdot)$ and $h_1 = c_N P_1^{\,c}(x,\cdot)$.
\end{proof}
Next, we let $s \in(\frac{1}{2},1)$. We use the following known bounds for Poisson kernels given in \cite[Theorem 2.10 and Remark 2.3]{C99}. There is $C=C(N,\Omega)>0$ such that,  for $s\in(\tfrac{1}{2},1)$,
\begin{equation}\label{P2}
P_s(z,y)\leq C\frac{(1-s) \delta^s(z)}{\delta^s(y)(1+\delta(y))^s|y-z|^N},\qquad
C^{-1}\frac{\delta(z)}{|y-z|^N}\leq P_1(z,y)\leq C\frac{\delta(z)}{|y-z|^N}.
\end{equation}
Using these bounds, we now derive the following uniform $L^1$-estimate for $P_s^{\,c}$.

\begin{lemma}\label{hnbd:lem}
There is a constant $C=C(N,\Omega)>0$ such that
\begin{align*}
    \|P_s^{\,c}(x,\cdot)\|_{L^1(\Omega)}  \leq C \delta^{-1}(x) \qquad \text{ for all $s \in (\frac{1}{2},1)$, $x\in\Omega$.}
\end{align*}
\end{lemma}

\begin{proof}
Let $x\in\Omega$. By \cite[Lemma 2.3]{JSW20} (with $a:=-s$ and $\lambda = 0$), there is $C_1=C_1(N,\Omega)>0$ such that
\begin{align}\label{z}
    \int_{\Omega}\frac{\delta^{s}(z)}{|y-z|^N}\, dz
    <C_1\frac{1}{(1+s)s}\frac{1}{1+\delta(y)^N}
    <\frac{2C_1}{1+\delta(y)^N}
    \qquad \text{ for all $n\in\N$ and $y\in \Omega^c.$}
\end{align}
Then, by \eqref{P2}, \eqref{z}, and Fubini's Theorem, there is $C_2=C_2(N,\Omega)>0$ such that
\begin{align}
    \|P_s^{\,c}(x,\cdot)\|_{L^1(\Omega)}&\leq
C_2(1-s)\int_{\Omega^c}\frac{1}{\delta^{s}(y)(1+\delta(y))^{s}|x-y|^{N}} \int_\Omega\frac{\delta^{s}(z)}{|y-z|^N}\,dz\, dy\notag\\
&\leq 2C_1C_2
(1-s)\int_{\Omega^c}\frac{1}{\delta^{s}(y)(1+\delta(y))^{s}|x-y|^{N}(1+\delta(y)^N)}\, dy.\label{h1}
\end{align}
Let $r_\Omega=\diam(\Omega)+1$, where, as usual, $\diam(\Omega):=\sup\{|x-y|\::\: x,y\in \Omega\}$. By \cite[Lemma 2.3]{JSW20} (with $a:=s>0=:\lambda$),
there is $C_3=C_3(N,\Omega)>0$ such that
\begin{align}
(1-s)\int_{B_{r_\Omega}(x)\setminus \Omega}
&\frac{|x-y|^{-N}}{\delta^{s}(y)(1+\delta(y))^{s}(1+\delta^N(y))}
\,dy \notag\\
&\leq (1-s)\int_{B_{r_\Omega}(x)\setminus \Omega}
\delta^{-s}(y)|x-y|^{-N}
\,dy\leq C_3\frac{1-s}{1-s}\left(1+\frac{\delta^{-s}(x)}{s}\right)\notag\\
&\leq 2C_3(1+\delta^{-s}(x)1_{\{\delta(x)>1\}}+\delta^{-s}(x)1_{\{\delta(x)\leq 1\}})
\leq 4C_3(1+\delta^{-1}(x)).\label{h2}
\end{align}
On the other hand, since $s\in(\frac{1}{2},1)$ and $\delta(y)>1$ for $y\in \R^N\setminus B_{r_\Omega}(x)$,
\begin{align}
(1-s)\int_{\R^N\setminus B_{r_\Omega}(x)}
\frac{|x-y|^{-N}}{\delta^{s}(y)(1+\delta(y))^{s}(1+\delta^N(y))}&
\,dy
\leq (1-s)r_\Omega^{-N}\int_{\R^N\setminus B_{r_\Omega}(x)}
\frac{1}{\delta^{2s+N}(y)}
\,dy \notag\\
&\leq r_\Omega^{-N}\sup_{\zeta\in \Omega}\int_{\R^N\setminus B_{r_\Omega}(\zeta)}
\frac{1}{\delta^{1+N}(y)}
\,dy =:C_4.\label{h3}
\end{align}
The claim now follows from \eqref{h1}, \eqref{h2}, and \eqref{h3}.
\end{proof}

With the help of the previous estimates, we may now derive convergence results related to a source term $f \in L^\infty(\Omega)$.

\begin{lemma}\label{hL2conv}
  Let $f \in L^\infty(\Omega)$, and let $\cPs_s^{\,c} f: \Omega \to \R$ be defined, for $s \in (0,1]$, by
  \begin{equation}
    \label{eq:def-cal-P-s-f}
  [\cPs_s^{\,c} f](x)= \int_{\Omega} P_s^{\,c}(x,z)f(z)\,dz.
  \end{equation}
  Then we have $\cPs_s^{\,c} f \to \cPs_1^{\,c} f$ as $s \to 1^-$ pointwisely on $\Omega$. Moreover,
  \begin{equation}
    \label{eq-lemma-hL2conv}
  \lim_{s \to 1^-}\|( \cPs_s^{\,c} f- \cPs_1^{\,c} f)\delta^a\|_{L^p(\Omega)}=0 \qquad \text{for any $p>1$ and $a\in(1-\frac{1}{p},1)$.}
  \end{equation}
\end{lemma}

\begin{proof}
  Since $f \in L^\infty(\Omega) \subset L^2(\Omega)$, we have, by Hölder's inequality and Corollary
\ref{pointwise:lem},
  $$
  | \cPs_s^{\,c} f (x)- \cPs_1^{\,c} f (x)| = \Bigl|  \int_{\Omega} [P_s^{\,c}(x,z)-P_1^{\,c}(x,z)] f(z)\,dz \Bigr|
  \le \|P_s^{\,c}(x,\cdot)-P_1^{\,c}(x,\cdot)\|_{L^2(\Omega)}\|f\|_{L^2(\Omega)} \to 0
  $$
  for $x \in \Omega$ as $s \to 1^-$, showing the pointwise convergence.\\
Next, let $p>1$ and $a\in(1-\frac{1}{p},1)$. By Lemma \ref{hnbd:lem}, there is $C_1=C_1(\|f\|_{L^\infty(\Omega)},\Omega)>0$ such that
\begin{align}\label{L1}
    \delta^{a}(x)|\cPs_s^{\,c} f (x)|\leq C_1(\delta^{a-1}(x)+1)=:L_1(x)\qquad \text{ for all }x\in\Omega,\ n\in \N.
\end{align}
Using \eqref{P2}, there is $C_2=C_2(\Omega)>0$ such that, for all $x\in\Omega$,
\begin{align}\label{hh1}
\int_{\partial \Omega}\frac{\delta(x)}{|x-y|^N}\,d \sigma(y)
\leq C_2\int_{\partial \Omega}P_1(x,y)\, d \sigma(y)=C_2,
\ \text{ which implies }\ 
    \int_{\partial \Omega}\frac{\delta^a(x)}{|x-y|^N}\, d \sigma(y)\leq C_2\delta^{a-1}(x).
\end{align}
Furthermore, by \eqref{z}, there is $C_3=C_3(\Omega)>0$ such that, for all $y\in \partial \Omega$,
\begin{align}\label{hh2}
\int_\Omega\frac{\delta(z)}{|y-z|^N}\, dz
\leq \sup_{\zeta\in\Omega}\delta^{1-s}(\zeta)\int_\Omega\frac{\delta^{s}(z)}{|y-z|^N}\, dz
\leq C_3(\diam(\Omega)+1).
\end{align}
Then, by \eqref{P2}, \eqref{hh1}, \eqref{hh2}, and Fubini's Theorem, there is $C_4=C_4(\|f\|_{L^\infty(\Omega)},\Omega)>0$ such that
\begin{align}\label{L2}
\delta^a(x)|\cPs_1^{\,c} f (x)|\leq C\|f\|_{L^\infty(\Omega)}
\int_{\partial \Omega}\frac{\delta^a(x)}{|x-y|^{N}}
\int_\Omega\frac{\delta(z)}{|y-z|^N}\, dz\,d \sigma(y)
\leq C_4\delta^{a-1}(x)=:L_2(x).
\end{align}
As a consequence, by \eqref{L1} and \eqref{L2},
\begin{align*}
\delta^{a}(x)|\cPs_s^{\,c} f(x)- \cPs_1^{\,c} f(x)|
\leq \delta^{a}(x)|\cPs_s^{\,c} f(x)|+\delta^{a}(x)|\cPs_1^{\,c} f(x)|\leq L_1(x)+L_2(x)=:L(x).
\end{align*}
Using that $p(a-1)+1>0$ (because $a>1-\frac{1}{p}$) we obtain that $\int_\Omega \delta^{p(a-1)}(x)\, dx<\infty$  (see the proof of \cite[Lemma 2.3]{JSW20}) and therefore $L\in L^p(\Omega)$.  Since $\cPs_s^{\,c} f(x) \to \cPs_1^{\,c} f(x)$ pointwisely in $\Omega$, the claim now follows by the dominated convergence theorem.
\end{proof}

\begin{lemma}\label{lem:con:2}
  Let $f\in C(\overline{\Omega})$, and let $\cPs_s^{\,c} f: \Omega \to \R$ be defined by (\ref{eq:def-cal-P-s-f}) for $s \in (0,1]$,
Then we have
\begin{align*}
\cGs_{s} (\cPs_s^{\,c} f) \to  \cGs_1(\cPs_1^{\,c} f)\qquad \text{uniformly in $\Omega$ as $s \to 1^-$.}
\end{align*}
\end{lemma}
\begin{proof}
  In the following, $C$ denotes possibly different constants depending at most on $\Omega$. Moreover, we put $h_s:= \cPs_s^{\,c} f$ for $s \in (0,1]$. We then note that
\begin{align*}
 |\cGs_{s}h_s-\cGs_1h_1|&\leq |\cGs_{s}(h_s-h_1)|+|(\cGs_{s}-\cGs_1)h_1|\quad\text{in $\Omega$.}
\end{align*}
By \cite[Theorem 2.4, eq. (2.14)]{C99},
\begin{align*}
|G_{s}(x,w)|\leq C\delta^{s}(w)|x-w|^{s-N}\qquad \text{ for all }x,w\in \Omega,\ x\neq w,\ s \ge \frac{1}{2}.
\end{align*}
Let $p\in(1,\frac{N}{N-\frac{1}{2}})$, $a\in(1-\frac{1}{p},1)$, and $q=\frac{p}{p-1}$. Then, by H\"older's inequality, for $x\in \Omega$ and $s>a$ we have
\begin{align*}
|\cGs_{s}(h_s-h_1)(x)|&\leq C\int_{\Omega}|x-w|^{s-N} \delta^{s}(w)|h_s(w)-h_1(w)|\ dw\\
&\leq C\Bigg(\int_{\Omega}|x-w|^{(s-N)q}\ dw\Bigg)^{\frac{1}{q}}\|(h_s-h_1)\delta^{a}\|_{L^p(\Omega)}.
\end{align*}
Then
\begin{align*}
\int_{\Omega}|x-w|^{(s-N)q}\ dw
&= \int_{\Omega\cap B_1(x)}|x-w|^{(s-N)q}\ dw
+\int_{\Omega\backslash B_1(x)}|x-w|^{(s-N)q}\ dw\\
&\leq C\int_0^1\rho^{(\frac{1}{2}-N)q+N-1}\ d\rho+|\Omega|<C,
\end{align*}
because $(\frac{1}{2}-N)q+N>0.$  Moreover, by Lemma \ref{hL2conv}, $\|(h_s-h)\delta^{a}\|_{L^p(\Omega)}\to 0$ and therefore
\begin{align*}
    \|\cGs_{s}(h_s-h_1)\|_{L^\infty(\Omega)}\to 0\qquad \text{ as } s\to 1^{-}.
\end{align*}
On the other hand, $ (\cGs_{s}-\cGs_1)h_1 \to 0$ as $s \to 1^{-}$ pointwisely in $\Omega$, by Lemma~\ref{lem:continues-with-blow-up}. This ends the proof.

\end{proof}

\section{Characterization of the derivatives of the solution map}\label{char:sec}

\subsection{Limit from below}\label{Sec:below}

\begin{prop}\label{prop:below}
For $f\in C^\alpha(\overline{\Omega})$, we have
\begin{equation}\label{ll}
\begin{split}
v_{1}&:=\lim_{\sigma\to 0^+}\frac{u_{1}-u_{1-\sigma}}{\sigma} =-\cGs_1 \left(\loglap E_\Omega f-\cPs_1^{\,c} f \right) \qquad \text{in $L^\infty(\Omega)$}
\end{split}
\end{equation}
and $v_s\to v_{1}$ in $L^\infty(\Omega)$ as $s\to 1^-$. Moreover, for all $\varphi\in C^\infty_c(\Omega)$,
\begin{align}
\int_\Omega v_{1^-}(x) (-\Delta)\varphi(x)\, dx
=\int_\Omega -\loglap E_\Omega f(x) \varphi(x)\, dx
+ \int_{\Omega}[\cPs_1 \loglap \varphi](x) f(x)\, dx\label{limit from below}
\end{align}
\end{prop}

\begin{proof}
  Let $s\in (0,1)$. By \cite[Theorem 1.1]{JSW20}, we know that, for $x\in\R^N,$
\begin{align*}
v_s
&= \partial_s u_s
= \cGs_s(-\loglap E_\Omega f - \loglap 1_{\R^N\backslash \Omega}(-\Delta)^s \cGs_s f),
\end{align*}
where, for $\rho \in \Omega$,
\begin{align*}
  &[\loglap 1_{\R^N\backslash \Omega}(-\Delta)^s \cGs_s f](\rho)= -c_N \int_{\R^N\backslash \Omega}\frac{[(-\Delta)^s \cGs_s f](y)}{|\rho-y|^N}dy \\
  &= -c_N \int_{\R^N\backslash \Omega}|\rho-y|^{-N} (-\Delta)^s \Bigl(\int_\Omega G_s(z,\cdot)f(z)\, dz\Bigr)(y)\,dy  \\
  &=  c_{N,s} c_N \int_{\R^N\backslash \Omega} |\rho-y|^{-N}   \int_\Omega \int_\Omega  G_s(z,\tau)f(z)\, dz\, |y-\tau|^{-N-2s} d\tau \, dy\\
  &=  c_{N,s} c_N \int_{\R^N\backslash \Omega}|\rho-y|^{-N}   \int_\Omega \int_\Omega  G_s(z,\tau)|y-\tau|^{-N-2s} d\tau f(z)\, dz  \, dy\\
  &=  - c_N \int_{\R^N\backslash \Omega}|\rho-y|^{-N}   \int_\Omega P_s(z,y) f(z)\, dz  \, dy =  - c_N \int_\Omega  \int_{\R^N\backslash \Omega}|\rho-y|^{-N} P_s(z,y)\,dy  f(z)\, dz\\
  &=  - \int_\Omega  P_s^{\,c}(\rho,z) f(z)\, dz = - [\cPs_s^{\,c} f](\rho).
  \end{align*}
  In this computation, we changed the order of integration multiple times, which is justified by the Fubini-Tonelli-Theorem if $f$ is nonnegative. In the general case, we split $f= f^+-f^-$ and note that the RHS of the computation above yields a finite value both for $f^+$ and $f^-$. As a consequence, we get
$$
v_s(x)= -\cGs_s(\loglap E_\Omega f- \cPs_s^{\,c} f ).
$$
Since $|L_{\Delta}E_{\Omega}f|\leq C(1+|\ln(\delta)|)$ for some $C>0$ (see \cite[Remark 4.3]{JSW20}), Lemma~\ref{lem:continues-with-blow-up} implies that
$$
    \cGs_s(\loglap E_\Omega f) \to \cGs_1(\loglap E_\Omega f)\qquad \text{in $L^\infty(\Omega)$ as $s\to 1^-$.}
$$
Furthermore, by Lemma \ref{lem:con:2}, we also have
$$
\cGs_s \left(\cPs_s^{\,c} f\right)  \to \cGs_1\left(\cPs_1^{\,c} f\right) \qquad \text{in $L^\infty(\Omega)$ as $s\to 1^-$.}
$$
Consequently, $\tilde v_1:= \lim \limits_{s \to 1^-}v_s$ exists in $L^\infty(\Omega)$ and is given by the RHS of (\ref{ll}). Since, moreover, $u_1 = \lim_{\tau \to 1^-}u_\tau$ in $L^\infty(\Omega)$ by Lemma~\ref{lem:continues-with-blow-up}, we have
  $$
  u_{1}-u_{1-\sigma} = \lim_{\tau \to 1^-}\bigl(u_\tau-u_{1-\sigma}\bigr)= \lim_{\tau \to 1^-} \int_{1-\sigma}^\tau v_s \,ds
  = \int_{1-\sigma}^1 v_s\,ds
  $$
  and
  $$
  \Bigl\| \frac{u_1-u_{1-\sigma}}{\sigma} - \tilde v_1   \Bigr\|_{L^\infty(\Omega)} =
  \frac{1}{\sigma} \Bigl\| \int_{1-\sigma}^1 (v_s -\tilde v_1) \,ds \Bigr\|_{L^\infty(\Omega)}\le
  \sup_{1-\sigma <s<1} \|v_s-\tilde v_1\|_{L^\infty(\Omega)} \to 0
  $$
  as $\sigma \to 0^+$. Consequently,
  $$
  v_1= \lim_{\sigma \to 0^+} \frac{u_1-u_{1-\sigma}}{\sigma} = \tilde v_1 \qquad \text{in $L^\infty(\Omega)$,}
  $$
  which gives (\ref{ll}). Finally, for $\varphi\in C^\infty_c(\Omega)$, we have
\begin{align*}
\int_\Omega v_{1^-}(-\Delta)\varphi\, dx
=\int_\Omega -\loglap f \varphi\, dx
-c_N\int_\Omega \int_\Omega \int_{\partial \Omega}\frac{f(z) P_1(z,y)\varphi(x)}{|x-y|^{N} }\,d \sigma(y)\, dz\, dx .
\end{align*}
Since
\begin{align*}
-c_N\int_\Omega \int_\Omega \int_{\partial \Omega}\frac{f(z) P_1(z,y)\varphi(x)}{|x-y|^{N} }\,d \sigma(y)\, dz\, dx
&=c_N\int_\Omega f(z) \int_{\partial \Omega} P_1(z,y)\int_\Omega \frac{\varphi(y)-\varphi(x)}{|x-y|^{N} }dx\,d \sigma(y)\, dz \\
&=\int_\Omega f(z) [\cPs_1 \loglap\varphi](z)\, dz,
\end{align*}
we obtain \eqref{limit from below}. The proof is finished.
\end{proof}

\subsection{Limit from above}\label{Sec:above}

In this section, we use the notation and assumptions of Theorem \ref{thm1:differential:2}. Let $f\in C^\alpha(\overline{\Omega})$, $s=1+\sigma$ for some $\sigma\in(0,1)$ and consider the solution of
\begin{align*}
    (-\Delta)^s u_s =f,\qquad u_s\in \cH^s_0(\Omega),
\end{align*}
where $(-\Delta)^s u_s =(-\Delta)(-\Delta)^{\sigma} u_s$.  For $x\in\Omega$, let
\begin{align*}
v_{1^+}(x):=\partial_s u_{s}(x) |_{s=1^+}
= \partial_\sigma u_{1+\sigma}(x) |_{\sigma=0^+}
=\lim_{\sigma\to 0^+}\frac{u_{1+\sigma}(x)-u_{1}(x)}{\sigma}.
\end{align*}
As mentioned in the introduction, for the following statement we need more regularity on the boundary of $\Omega$ to guarantee the regularity of solutions based on the results by \cite{G15}. In view of the higher-order regularity in Sobolev spaces obtained in \cite{BN23}, we believe, however, that the regularity assumption on the boundary can be weakened significantly, though we do not pursue this here.

\begin{prop}\label{prop:above}
Assume $\partial \Omega$ is of class $C^{\infty}$. For $f\in C^\alpha(\overline{\Omega})$, we have that
\begin{align*}
v_{1^+}(x)=[\cGs_1\Delta \loglap [\cGs_1 f]](x),\qquad x\in\Omega,
\end{align*}
and, for all $\varphi\in C^\infty_c(\Omega)$,
\begin{align}
\int_\Omega v_{1^+}(x) (-\Delta)\varphi(x)\, dx
=\int_\Omega -\loglap E_\Omega f(x) \varphi(x)\, dx
+ [\cPs_1 \loglap \varphi](x) f(x)\, dx.\label{limit from above}
\end{align}
\end{prop}
\begin{proof}
By regularity theory, we have that $(-\Delta)^{\sigma}u_{1+\sigma}\in C^{1-\sigma}(\R^N)$ and $(-\Delta)^{\sigma}u_{1+\sigma}\big|_{\Omega}\in C^{2+\epsilon}_{loc}(\Omega)$ for some $\epsilon\in(0,\alpha)$, see e.g. \cite{G15}. It follows that
\begin{align*}
    (-\Delta)^s u_s=(-\Delta)(-\Delta)^\sigma u_{1+\sigma}=f\quad \text{ in }\Omega.
\end{align*}
Using the representation formulas for the Laplacian, we obtain that
\begin{align}\label{1}
(-\Delta)^\sigma u_{1+\sigma} = \cGs_1 f + \cPs_1 (-\Delta)^\sigma u_{1+\sigma}\qquad \text{ in }\Omega.
\end{align}
Since $u_{1+\sigma}\in \cH^{1+\sigma}_0(\Omega)\subset \cH^\sigma_0(\Omega)$ and the right-hand side of \eqref{1} is H\"older continuous up to the boundary, we can write \eqref{1} as
\begin{align*}
     u_{1+\sigma} = \cGs_\sigma \cGs_1 f + \cGs_\sigma \cPs_1 (-\Delta)^\sigma u_{1+\sigma}.
\end{align*}
By \cite[Proposition 4.1]{JSW20}, for $\beta>0$ and $g\in C^\beta(\overline{\Omega})$, we have that
\begin{align*}
    \partial_\sigma \cGs_\sigma g\mid_{\sigma=0} =
    \lim_{\sigma\to 0^+}\frac{\cGs_\sigma-I}{\sigma}g=
    -\loglap E_\Omega g.
\end{align*}
Moreover, by \cite[Theorem 1.1]{CW19},  for $g\in C_c^\beta(\overline{\Omega})$,
\begin{align*}
    \partial_\sigma (-\Delta)^\sigma g\mid_{\sigma=0}
    =\lim_{\sigma\to 0^+}\frac{(-\Delta)^\sigma -I }{\sigma}g
    = \loglap g.
\end{align*}
In particular, $\partial_\sigma \cGs_\sigma \cGs_1 f\mid_{\sigma=0} =-\loglap \cGs_1 f$ and, since $\cGs_1 f=0$ in $\R^N\backslash \Omega$ implies that $\cPs_1u_1=\cPs_1 \cGs_1 f\equiv 0$,
\begin{align*}
\partial_\sigma (\cGs_\sigma \cPs_1 (-\Delta)^\sigma u_{1+\sigma})\mid_{\sigma=0}
&=
\lim_{\sigma\to 0^+}\frac{
\cGs_\sigma \cPs_1 (-\Delta)^\sigma u_{1+\sigma}-\cPs_1 u_{1}
}{\sigma}\\
&=\lim_{\sigma\to 0^+}\frac{
(\cGs_\sigma-I)}{\sigma} \cPs_1 (-\Delta)^\sigma u_{1+\sigma}
+\cPs_1 \frac{(-\Delta)^\sigma u_{1+\sigma}}{\sigma}=\cPs_1 \loglap \cGs_1 f.
\end{align*}

Therefore,
\begin{align*}
v_{1^+}:=\partial_s u_{s} |_{s=1^+}
= \partial_\sigma u_{1+\sigma} |_{\sigma=0^+}
=-\loglap \cGs_1 f + \cPs_1 \loglap \cGs_1 f
=(I-\cPs_1)(-\loglap \cGs_1 f)
=\cGs_1(\Delta \loglap \cGs_1 f).
\end{align*}
Observe that $\cPs_1 \loglap \cGs_1 f$ is harmonic in $\Omega$. Then, if $\varphi\in C^\infty_c(\Omega)$,
\begin{align}
\int_\Omega v_{1^+} (-\Delta)\varphi\, dx
=\int_\Omega -\loglap \cGs_1 f (-\Delta)\varphi\, dx
=\int_\Omega (\cGs_1 f) \loglap \Delta \varphi\, dx
=\int_\Omega (\cGs_1 f) \Delta \loglap\varphi\, dx,\label{l}
\end{align}
where we used that $\loglap \Delta \varphi = \Delta \loglap  \varphi $ in $\R^N$ (this can be seen using the Fourier transform, for instance). Now, integrating by parts in \eqref{l},
\begin{align*}
\int_\Omega v_{1^+} (-\Delta)\varphi\, dx
&=-\int_\Omega \loglap f \varphi\, dx
-\int_{\partial \Omega}\partial_{\nu_\theta}\left(\int_{\Omega} G_1(\theta,z) f(z)\, dz\right) \loglap \varphi(\theta)\, d \sigma(\theta) \\
&=-\int_\Omega \loglap f \varphi\, dx
+\int_{\partial \Omega}\int_{\Omega} P_1(z,\theta) f(z)\, dz\, \loglap \varphi(\theta)\, d \sigma(\theta)\\
&=\int_\Omega -\loglap f \varphi
+ \cPs_1 \loglap \varphi(x) f(x)\, dx,
\end{align*}
as claimed.
\end{proof}

We may now complete the proof of our main result.

\begin{proof}[Proof of Theorem \ref{thm1:differential}]
The claim for $s=1$ follows from Proposition \ref{prop:below} and Lemma~\ref{lem:interchange}.

For $s\in(0,1)$ the differentiability is shown in \cite[Theorem 1.1]{JSW20} and also that \eqref{char-derivative-old} holds. Let $w_s$ as in \eqref{char-derivative-old}. For $x\in \Omega$,
\begin{align}
  L_{\Delta}w_s(x)&=c_N\int_{ \Omega^c} c_{N,s}\int_{\Omega}\frac{u(y)}{|x-z|^N|z-y|^{N+2s}}\, dydz=c_N\int_{\Omega^c} c_{N,s}\int_{\Omega}\int_{\Omega}\frac{G_s(y,w)f(w)}{|x-z|^N|z-y|^{N+2s}}\, dw dy dz\notag\\
   &=c_N\int_{\Omega}\int_{\Omega^c}\frac{f(w)}{|x-z|^N} \int_{\Omega} \frac{c_{N,s}G_s(y,w) }{ |z-y|^{N+2s}}\, dy dz dw=c_N\int_{\Omega}\int_{\Omega^c}\frac{f(w)P_s(w,z)}{|x-z|^N}\, dz dw,
   \label{lll}
\end{align}
where the last step follows directly from the definition of the Poisson kernel. The claim for $s\in(0,1)$ now follows from \eqref{char-derivative-old}, \eqref{lll}, and the fact that for $x\in \Omega$, by Lemma~\ref{lem:interchange} we have
\begin{align*}
(-\Delta)^sL_{\Delta} u(x)&=L_{\Delta} (-\Delta)^s u(x)=L_{\Delta}E_{\Omega} (-\Delta)^s u(x)-c_N\int_{\R^N\setminus\Omega} \frac{(-\Delta)^s u(y)}{|x-y|^N}\ dy,
\end{align*}
and the claim follows by \eqref{Psdef}, because, for $y\in \Omega^c$,
$$
-\int_{\Omega}P_s(z,y)(-\Delta)^su(z)\ dz
=(-\Delta)_y^s \int_{\Omega} G_s(z,y)(-\Delta)^su(z)\ dz
=(-\Delta)^s u(y)
.
$$
\end{proof}

\begin{proof}[Proof of Theorem \ref{thm1:differential:2}]
The claim follows from Proposition \ref{prop:above}. Note that, by \eqref{limit from below} and \eqref{limit from above}, we have that
\begin{align*}
    \int_\Omega v_{1^-} (-\Delta)\varphi\, dx =\int_\Omega v_{1^+} (-\Delta)\varphi\, dx \qquad \text{ for all }\varphi\in C^\infty_c(\Omega).
\end{align*}
Therefore, $v_{1^-}$ and $v_{1^+}$ satisfy Dirichlet boundary conditions on $\partial \Omega$ and solve distributionally the same problem. The uniqueness of solutions yields that $v_{1^-}=v_{1^+}$ in $\Omega$.
\end{proof}

\begin{proof}[Proof of Corollary \ref{cor:expansion at 1}]
Note that, for $x\in\Omega$,
\begin{align*}
u_1(x)-u_s(x) = \int_s^1 v_t(x)\, dt =
(1-s)v_{1}(x)+\int_s^1 v_t(x)-v_1(x)\, dt.
\end{align*}
The claim now follows, since $\|v_t-v_1\|_{L^\infty(\Omega)}\to 0$ as $t\to 1^-$, by Proposition \ref{prop:below}.
\end{proof}

\begin{remark}[On the boundary regularity of $v_1$]\label{rmk:boundary}
We say that $g_1\asymp g_2$ if there is $c>0$ such that $c^{-1}g_2\leq g_1\leq cg_2$. For $x\in \Omega$ with $\delta(x)<\frac12$, note that
\begin{align*}
&\int_\Omega G_{1}(w,x)\int_\Omega\int_{\partial \Omega}P_1(z,y)|x-y|^{-N}\,d \sigma(y)\, dz\, dx
\asymp \int_\Omega \frac{G_{1}(w,x)}{\delta(x)}
\int_{\partial \Omega}P_1(x,y)\int_\Omega P_1(z,y)\, dz\,d \sigma(y)\, dx\\
&\asymp \int_\Omega \frac{G_{1}(w,x)}{\delta(x)}
\int_{\partial \Omega}P_1(x,y)\,d \sigma(y)\, dx
=\int_\Omega \frac{G_{1}(w,x)}{\delta(x)}\, dx\asymp \delta(x)\ln\left(\frac{1}{\delta(x)}\right),
\end{align*}
where we used that $\int_{\partial \Omega}P_1(x,y)\,d \sigma(y)=1$, that $y\mapsto\int_\Omega P_1(z,y)\, dz\asymp 1$ in $\partial \Omega$ (by \eqref{P2} and \eqref{hh2}), and \cite[Proposition 3]{A15} for the last estimate (note that the proof of Proposition 3 in \cite{A15} easily extends to $s=1$ with minor changes).

As a consequence, by \eqref{ll}, if $f\in C^\alpha(\overline{\Omega})$ satisfies for example that $f\asymp 1$ and $\loglap E_{\Omega}f\gneq 0$, then
\begin{align}\label{v1b}
-v_1\asymp \delta(x)\left(1+|\ln\left(\delta(x)\right)|\right) \quad\text{in $\Omega$.}
\end{align}
This is consistent with the case of the fractional torsion function for balls, or more generally, for ellipsoids, which has an explicit solution. Indeed, if $N \geq 2,$ $s>0,$ $A \in \mathbb{R}^{N \times N}$ is a positive definite symmetric matrix, and $E:=\left\{x \in \mathbb{R}^n: A x \cdot x<1\right\}$, then $u_s(x):=c_{s,A}(1-A x \cdot x)_{+}^s$ solves $(-\Delta)^s u_s=1$ in $E$ for some explicit $c_{s,A}>0$, see \cite[Theorem~1.1]{AJS20}.  In this case,
\begin{align*}
v_1(x)=
\partial_s[c_{s,A}]|_{s=1}(1-A x \cdot x)
+c_{1,A}(1-A x \cdot x)\ln[(1-A x \cdot x)]\qquad \text{ for }x\in E,
\end{align*}
which has the behavior \eqref{v1b}.
\end{remark}

\section{Bounds on the Green operator}
\label{sec:bounds-green-oper}

As before, we let $\Omega \subset \R^N$, $N \ge 2$ be a bounded domain of class $C^2$. In the following, we restrict our attention on the (fractional) torsion problem, i.e., the case $f \equiv 1$ in $(D_s)$. We put
$$
u_s:= \cGs_s 1 \in L^\infty(\Omega) \qquad \text{for $s \in (0,1]$,}
$$
where $\cGs_s$ is the Green operator defined in \eqref{eq:G-op-definition}. We recall the following simple but important observation which immediately follows from the fact that $\cGs_s$ is an order preserving operator: {\em The maximum of the positive function $u_s$ equals the operator norm $\|\cGs_s\|$ of $\cGs_s$ when considered as a linear bounded operator
  $L^\infty(\Omega) \to L^\infty(\Omega)$.}

Hence, upper bounds on this number are of intrinsic interest. To derive these upper bounds, we use, similarly as in \cite{JSW20}, a pointwise differential inequality for the function $u_s$ with respect to the parameter $s$. By Theorem~\ref{thm1:differential}, we have
$$
-\partial_s u_s = \cGs_s (\loglap \cEO 1+ \cPs_s^{\,c} 1) \qquad \text{for $s \in (0,1]$,}
$$
where $\cPs_s^{\,c}$ is defined in (\ref{eq:def-cal-P-s-f}). Moreover, by (\ref{omegarep}) we have
$$
\loglap \cEO f = \loglap 1_\Omega = h_\Omega + \rho_N \qquad \text{in $\Omega$.}
$$
Setting
$$
m_s(\Omega) = \rho_N + \inf_{x \in \Omega}\Bigl(h_\Omega(x)+ \cPs_s^{\,c} 1\Bigr)\qquad \text{for $s \in (0,1]$}
$$
and using that the operator $\cGs_s$ is order preserving for $s \in (0,1]$, we therefore find that
\begin{align*}
-\partial_s u_s &= \cGs_s(\loglap 1_\Omega + \cPs_s^{\,c} 1)=  \cGs_s(h_\Omega+\rho_N + \cPs_s^{\,c} 1) \ge  \cGs_s(m(\Omega)1_{\Omega}) =  m(\Omega) \cGs_s(1_{\Omega})= m_s(\Omega) u_s
\end{align*}
for $s \in (0,1]$. Since moreover
$$
\lim_{s \to 0^+}u_s(x) = 1 \qquad \text{for every $x \in \Omega$}
$$
(see e.g. \cite[Theorem 1.3]{JSW20}), we find that
$$
u_s(x) = 1 + \int_0^s \partial_t u_t \,dt  \le 1 - \int_0^s m_t(\Omega) u_t(x)\,dt \qquad \text{for $x \in \Omega$, $s \in (0,1]$.}
$$
Therefore, by Grönwall's inequality,
$$
u_s \le \exp\Bigl(- \int_0^s m_\tau (\Omega)\,d\tau\Bigr) \quad \text{in $\Omega\qquad $ for $s \in (0,1]$,}
$$
which in turn gives
\begin{equation}
  \label{eq:m-tau-s-est}
\|\cGs_s\| \le \exp\Bigl(- \int_0^s m_\tau (\Omega)\,d\tau\Bigr)\qquad \text{for $s \in (0,1]$.}
\end{equation}
In \cite[Corollary 1.7]{JSW20}, we used the nonnegativity of the function $\cPs_s^{\,c} 1$ to simply estimate $m_s(\Omega) \ge \min_{\Omega}h_\Omega + \rho_N$ independently of $s$, which by (\ref{eq:m-tau-s-est}) gives
\begin{equation}
  \label{eq:m-tau-s-est-old}
\|\cGs_s\| \le \exp \Bigl(-s (\min_{\Omega}h_\Omega +\rho_N)\Bigr).
\end{equation}
In the remainder of this section, we would like to improve this estimate by deriving a positive lower bound for
\begin{equation}
  \label{eq:def-p-s-omega}
p_s(\Omega):= \inf_{\Omega}\, \cPs_s^{\,c} 1, \qquad s \in (0,1].
\end{equation}
This directly yields an improvement of (\ref{eq:m-tau-s-est-old}), since $m_s(\Omega) \ge \min \limits_{\Omega}h_\Omega + \rho_N + p_s(\Omega)$ and therefore
\begin{equation}
  \label{eq:m-tau-s-est-new-1}
\|\cGs_s\| \le \exp \Bigl(-s (\min_{\Omega}h_\Omega +\rho_N)-\int_{0}^s p_\tau(\Omega)\,d\tau\Bigr).
\end{equation}
To estimate $p_s$, we recall that
$$
[\cPs_s^{\,c} 1](x)= c_N \int_{\Omega} P_s^{\,c}(x,z)\,dz \qquad \text{for $x \in \Omega$, $s \in (0,1]$}.
$$
In the case $s=1$, we therefore easily get the following pointwise lower bound for the function $\cPs_1^{\,c} 1$ on $\Omega$:
\begin{align}
[\cPs_1^{\,c} 1](x)= c_N\int_{\Omega}\int_{\partial \Omega} P_1(z,y)|x-y|^{-N}\,d \sigma(y) dz
                     &\ge c_N (\diam\, \Omega)^{-N}\int_{\Omega}\int_{\partial \Omega} P_1(z,y)\,d \sigma(y) dz \nonumber \\
&= c_N |\Omega|(\diam\, \Omega)^{-N}\qquad \text{for $x \in \Omega$}.\label{P-1-c-lower-bound}
\end{align}
Here, we used the fact that
$$
\int_{\partial \Omega} P_1(z,y)\,d \sigma(y) = 1 \qquad \text{for every $z \in \Omega$.}
$$
By (\ref{P-1-c-lower-bound}), we have
\begin{equation}
  \label{eq:p-1-lower-bound}
p_1(\Omega) = \inf_{\Omega}\cPs_1^{\,c} 1 \ge c_N |\Omega|(\diam\, \Omega)^{-N}.
\end{equation}
In the case $s \in (0,1)$, we wish to bound $p_s(\Omega)$ similarly from below by an $s$-dependent multiple of $|\Omega|(\diam\, \Omega)^{-N}$, but this turns out to be much more difficult. We have the following explicit, but probably not optimal lower bound.

\begin{lemma}
  \label{lower-bound-p-function}
For $s \in (0,1)$ we have
  \begin{equation}
    \label{eq:lower-bound-p-function}
  p_s(\Omega) = \inf_{\Omega}\cPs_s^{\,c} 1 \ge
  c_N
  \left(\frac{3^s\Gamma(\frac{N}{2})}{{2^{N}}\Gamma(s)\Gamma(\frac{N}{2}+1-s)} \right)^{\frac{N}{N-2s}} |\Omega| \diam(\Omega)^{-N}.
  \end{equation}
\end{lemma}
\begin{proof}
We consider the nonnegative function
$$
\tau: \Omega^c \to \R, \qquad \tau(y) = \int_{\Omega}P_s(z,y)\,dz,
$$
which satisfies
\begin{equation}
  \label{eq:normalization-tau}
\int_{\Omega^c}\tau(y)\,dy = \int_{\Omega} \int_{\Omega^c}P_s(z,y)\,dy dz = \int_{\Omega}1 dz = |\Omega|.
\end{equation}
By (\ref{eq:normalization-tau}) and Jensen's inequality, we have
$$
\int_{\Omega^c}|x-y|^{-N} \tau(y)\,dy \ge |\Omega|^{-\frac{2s}{N-2s}} \Bigl(\int_{\Omega^c}|x-y|^{2s-N} \tau(y)\,dy\Bigr)^{\frac{N}{N-2s}}.
$$
Let $e_{N,s}:=  c_N |\Omega|^{-\frac{2s}{N-2s}} \kappa_{N,s}^{-\frac{N}{N-2s}}$, where
\begin{align*}
\kappa_{N, s}:=\frac{\Gamma\left(\frac{N}{2}-s\right)}{4^s \pi^{N / 2} \Gamma(s)}=\frac{s \Gamma\left(\frac{N}{2}-s\right)}{4^s \pi^{N / 2} \Gamma(1+s)}.
\end{align*}
In particular, $F_s(z):=\kappa_{N, s}|z|^{2 s-N}$ is the fundamental solution of $(-\Delta)^s$.  Then
\begin{align*}
&[\cPs_s^{\,c} 1](x) \ge e_{N,s}
\Bigl(\int_{\Omega^c}F_s(x-y)\tau(y)\,dy\Bigr)^{\frac{N}{N-2s}}\\
&= e_{N,s}
\Bigl(\int_{\Omega} \int_{\Omega^c}F_s(x-y)P_s(z,y)\,dy dz\Bigr)^{\frac{N}{N-2s}}= e_{N,s}
                \Bigl(\int_{\Omega} \Bigl(F_s(x-z)- G_s(x,z)\Bigr) dz\Bigr)^{\frac{N}{N-2s}}.
\end{align*}
Here, we used in the last step that, for $x \in \Omega$, the function $z \mapsto F_s(x-z)- G_s(x,z)$ is the unique $s$-harmonic function in $\Omega$ which is bounded in a neighborhood of $\partial \Omega$ and coincides with $F_s(x-\cdot)$ on $\Omega^c$.
By translation, we may assume in the following that $\Omega \subset B_r(0)$, where $r:= \diam \Omega$. Let $B = B_{2r}(0)$. By the maximum principle, we then have
\begin{equation}
  \label{eq:green-function-mono}
G_s(x,z ) \le G_s'(x,z) \qquad \text{for every $x,y \in \Omega \subset B$}
\end{equation}
where $G_s'(x,z)$ denotes the Green function associated with $B$.  Moreover, we write $P_s{'}$ for the Poisson kernel of the ball $B$.
Using (\ref{eq:green-function-mono}) in the above estimate, we obtain, for $x \in \Omega$, that
\begin{align*}
&[\cPs_s^{\,c} 1](x) \ge e_{N,s}
\Bigl(\int_{\Omega} \Bigl(F_s(x-z)- G_s'(x,z)\Bigr) dz\Bigr)^{\frac{N}{N-2s}}\\
&= e_{N,s}
\Bigl(\int_{\Omega} \int_{B^c}F_s(x-y)  P_s'(z,y) dy dz\Bigr)^{\frac{N}{N-2s}}= e_{N,s}
\Bigl(\kappa_{N,s} \int_{B^c} |x-y|^{2s-N}  \int_{\Omega}  P_s'(z,y) dz dy\Bigr)^{\frac{N}{N-2s}}\\
&\ge e_{N,s}
\Bigl(\kappa_{N,s} \int_{B^c}\bigl(2 |y|\bigr)^{2s-N} \int_{\Omega} P_s'(z,y) dzdy\Bigr)^{\frac{N}{N-2s}}= c_N  |\Omega|^{-\frac{2s}{N-2s}}2^{-N} \Bigl(\int_{B^c}|y|^{2s-N} \int_{\Omega} P_s'(z,y) dzdy\Bigr)^{\frac{N}{N-2s}}.
\end{align*}
Here, we used that
$$
|x-y|^{2s-N} \ge (|x|+|y|)^{2s-N} \ge (2|y|)^{2s-N}\qquad \text{for $x \in \Omega \subset B_r(0)$, $y \in B^c$.}
$$
Now we note that $P_s'$ is explicitly given by
\[
  P_s'(z,y)= \tau_{N,s} \frac{((2r)^2-|z|^2)^s}{(|y|^2-(2r)^2)^s}|z- y|^{-N}
  \qquad \text{for $z\in B,\ y \in B^c$}
 \]
 with
 $$
 \tau_{N,s} = \frac{2}{\Gamma(s)\Gamma(1-s)|S^{N-1}|} =  \frac{c_N}{\Gamma(s) \Gamma(1-s)}.
 $$
Consequently, we find that
$$
[\cPs_s^{\,c} 1](x) \ge
c_N  |\Omega|^{-\frac{2s}{N-2s}} \tau_{N,s}^{\frac{N}{N-2s}}2^{-N}
  \Bigl(\int_{B^c}|y|^{2s-N}(|y|^2-(2r)^2)^{-s} \int_{\Omega} ((2r)^2-|z|^2)^{s}|z-y|^{-N} dzdy\Bigr)^{\frac{N}{N-2s}}
  $$
We now use the facts that for $y \in B^c$ and $z \in \Omega$ we have
  $$
 ((2r)^2-|z|^2)^{s} \ge ((2r)^2-r^2)^{s}= 3^sr^{2s}
 \quad\text{and}\quad |z-y|^{-N}\ge  (|y|+r)^{-N} \ge (2|y|)^{-N}.
  $$
  Then, for $y\in B^c$,
  \begin{align*}
      \int_{\Omega} ((2r)^2-|z|^2)^{s}|z-y|^{-N} dz\geq 3^sr^{2s}\int_{\Omega} (2|y|)^{-N} \, dz
      =3^sr^{2s}|\Omega|2^{-N}|y|^{-N}.
  \end{align*}
  Thus, using polar coordinates and that $c_N=2/|\mathbb S^{N-1}|$,
	\begin{align*}
	[\cPs_s^{\,c} 1](x) &\ge
c_N   2^{-N}
  \Bigl( \tau_{N,s} 3^{s}r^{2s}{2^{-N}}|\Omega| \int_{B^c}|y|^{2s-2N}(|y|^2-(2r)^2)^{-s} dy \Bigr)^{\frac{N}{N-2s}} |\Omega|^{-\frac{2s}{N-2s}}\\
	&=
c_N  2^{-N}
  \Bigl( \tau_{N,s} 3^{s}r^{2s} \frac{{2^{1-N}}}{c_N}  \int_{2r}^{\infty} t^{2s-N-1}(t^2-(2r)^2)^{-s} dt \Bigr)^{\frac{N}{N-2s}} |\Omega|\\
	&=
c_N
  \Bigl(3^s  \frac{\tau_{N,s}}{c_N} {2^{1-N}}r^{2s} (2r)^{-N}  \int_{1}^{\infty} \rho^{2s-N-1}(\rho^2-1)^{-s} d\rho \Bigr)^{\frac{N}{N-2s}} |\Omega|\\
		&=
c_N \left(3^s {2^{1-N}} \right)^{\frac{N}{N-2s}}
  \Bigl( \frac{\tau_{N,s}}{c_N}  \int_{1}^{\infty} \rho^{2s-N-1}(\rho^2-1)^{-s} d\rho \Bigr)^{\frac{N}{N-2s}}  |\Omega|  r^{-N} \\
			&=
c_N \left(3^s {2^{1-N}} \right)^{\frac{N}{N-2s}}
  \Bigl( \frac{1}{\Gamma(s)\Gamma(1-s)}  \frac{\Gamma(1-s)\Gamma(\frac{N}{2})}{{2}\Gamma(\frac{N}{2}+1-s)} \Bigr)^{\frac{N}{N-2s}}|\Omega|  r^{-N}\\
				&=
c_N
  \left(\frac{3^s\Gamma(\frac{N}{2})}{{2^{N}}\Gamma(s)\Gamma(\frac{N}{2}+1-s)} \right)^{\frac{N}{N-2s}} |\Omega|r^{-N},
	\end{align*}
as claimed.

\end{proof}
Combining (\ref{eq:m-tau-s-est-new-1}) and Lemma~\ref{lower-bound-p-function}, we immediately get the following improvement of
\cite[Corollary 1.7]{JSW20}.

\begin{thm}
  \label{new-theorem-bound-green}
For $s \in (0,1]$, we have
\begin{equation}
  \label{eq:m-tau-s-est-new-2}
\|\cGs_s\| \le \exp \Bigl(-s (\min_{\Omega}h_\Omega +\rho_N)- q_{N,s} |\Omega| \diam(\Omega)^{-N}  \Bigr)
\end{equation}
with
$$
{q_{N,s} = c_N\int_{0}^s \left(\frac{3^\tau\Gamma(\frac{N}{2})}{{2^{N}}\Gamma(\tau)\Gamma(\frac{N}{2}+1-\tau)} \right)^{\frac{N}{N-2\tau}} d\tau.}
$$
\end{thm}

We remark that bounds from below for $\min_{\Omega}h_\Omega $ in terms of the geometry of $\Omega$ are also available, see \cite[eq. (1.17)]{JSW20} .

\appendix

\section{Auxiliary results of independent interest}

In this section we include some auxiliary results that could be of independent interest. In particular, we include the proof of Lemma~\ref{lem:interchange} which has been used in the proof of Theorem~\ref{thm1:differential}. Let
$$
\cL^1_0:=\Big\{f\in L^1_{loc}(\R^N)\;:\; \int_{\R^N}\frac{|f(x)|}{1+|x|^N}\ dx<\infty\Big\}.
$$

\begin{lemma}\label{lem:ibyp}
Let $\alpha>0$ and $u,v\in C^{\alpha}(\R^N)\cap \cL^1_0$. Then
$$
\int_{\R^N}L_{\Delta}u(x)v(x)\ dx=\int_{\R^N} u(x)L_{\Delta}v(x)\ dx.
$$
The statement also holds if $u=E_{\Omega}f$ or $v=E_{\Omega}g$ with $f,g\in C^{\alpha}(\overline{\Omega})$. Moreover, for $g\in C^{\alpha}(\overline{\Omega})$,
$$
\int_{\Omega} \loglap E_{\Omega^c}u(x)g(x)\ dx=\int_{\Omega^c} u(y)L_{\Delta}E_{\Omega}g(y)\ dy.
$$
\end{lemma}
\begin{proof}
By assumption we have that the quantity
$\int_{\R^N} (h_{\Omega}(x)+\rho_N)u(x)v(x) dx
$
exists and it is symmetric in $u$ and $v$. Thus,
\begin{align*}
\frac{1}{c_N}&\int_{\R^N}L_{\Delta}u(x)v(x)-u(x)L_{\Delta} v(x)\ dx\\
&=\int_{\R^N}v(x)\int_{\Omega}\frac{u(x)-u(y)}{|x-y|^N}\ dydx -\int_{\R^N}\int_{\Omega^c}\frac{v(x)u(y)}{|x-y|^N}\ dydx\\
&\quad -\int_{\R^N}u(x)\int_{\Omega}\frac{v(x)-v(y)}{|x-y|^N}\ dydx +\int_{\R^N}\int_{\Omega^c}\frac{u(x)v(y)}{|x-y|^N}\ dydx\\
&=\int_{\Omega}\int_{\Omega} \frac{v(x)(u(x)-u(y))-u(x)(v(x)-v(y))}{|x-y|^N}\ dydx \\
&\quad + \int_{\Omega^c} \int_{\Omega}\frac{v(x)(u(x)-u(y))-u(x)(v(x)-v(y))}{|x-y|^N}\ dydx\\
&\quad + \int_{\Omega}\int_{\Omega^c}\frac{u(y)v(x)-v(y)u(x)}{|x-y|^N}\ dxdy+ \int_{\Omega^c}\int_{\Omega^c}\frac{u(x)v(y)-v(x)u(y)}{|x-y|^N}\ dydx\\
&=\int_{\Omega}\int_{\Omega} \frac{u(x)v(y)-v(x)u(y)}{|x-y|^N}\ dydx \\
&\quad + \int_{\Omega^c} \int_{\Omega}\frac{u(x)v(y)-v(x)u(y)}{|x-y|^N}\ dydx + \int_{\Omega^c}\int_{\Omega}\frac{u(y)v(x)-v(y)u(x)}{|x-y|^N}\ dydx=0.
\end{align*}
The first additional assertion follows by approximating the kernel $|\cdot|^{-N}$ with $\min\{n,|\cdot|^{-N}\}$ and $n\in\N$ in the above calculation and the second assertion follows by using similar (simpler) arguments.
\end{proof}

\begin{proof}[Proof of Lemma \ref{lem:interchange}]
Let $u$ as in the statement and let $\phi\in C^{\infty}_c(\Omega)$. If $s=1$, then, using that $u=0$ on $\partial \Omega$ and Lemma \ref{lem:ibyp},
\begin{align}
\int_{\Omega} \Delta L_{\Delta}u(x)\phi(x)\ dx
&=\int_{\Omega}u(x) L_{\Delta}\Delta \phi(x)\ dx=\int_{\Omega}u(x) \Delta L_{\Delta} \phi(x)\ dx\notag\\
&=\int_{\Omega}\Delta u(x) L_{\Delta} \phi(x)\ dx
+\int_{\partial \Omega}u(z)\partial_{\eta} L_{\Delta} \phi(z)-\partial_{\eta}u(z) L_{\Delta} \phi(z)\ d \sigma(z)\notag\\
&=\int_{\Omega}\Delta u(x) L_{\Delta} \phi(x)\ dx
-\int_{\partial \Omega}\partial_{\eta}u(z) L_{\Delta} \phi(z)\ d \sigma(z),\label{i1}
\end{align}
where $\eta$ is the exterior unit vector at $z\in\partial \Omega$. Furthermore, since $\phi=0$ in a neighborhood of $\partial\Omega$,
\begin{align}
c_N\int_{\Omega}&\int_{\Omega} \int_{\partial \Omega} \frac{P_1(z,y) \Delta u (z)\phi(x)}{|x-y|^{N}} \,d \sigma(y)dz\,dx
=-c_N\int_{\partial \Omega}\int_{\Omega}P_1(z,y) \Delta u (z)\int_{\Omega}  \frac{\phi(y)-\phi(x)}{|x-y|^{N}}\,dx\,dz\,d \sigma(y)\notag\\
&=-\int_{\partial \Omega}\int_{\Omega}P_1(z,y) \Delta u (z)L_{\Delta}\phi(y)\,dzd \sigma(y)
=-\int_{\Omega}\Delta u(z)\cPs_1 L_{\Delta}\phi(z)\ dz\notag\\
&=\int_{\partial \Omega}u(z)\partial_{\eta} [\cPs_1 L_{\Delta}\phi](z) -\partial_{\eta}u(z) L_{\Delta} \phi(z)\,d \sigma(z)
=-\int_{\partial \Omega}\partial_{\eta}u(z) L_{\Delta} \phi(z)\,d \sigma(z),\label{i2}
\end{align}
By \eqref{i1}, \eqref{i2}, and Lemma \ref{lem:ibyp}, we find that
$$
\int_{\Omega}\Delta L_{\Delta}u(x)\phi(x)\ dx=\int_{\Omega}\phi(x)\Bigg(L_{\Delta} \Delta u(x)
+c_N \int_{\Omega} \int_{\partial \Omega} \frac{P_1(z,y) \Delta u (z)}{|x-y|^{N}} \,d \sigma(y)dz\Bigg)\ dx
$$
and the claim follows for $s=1$ from the fundamental lemma of calculus of variations.

Now assume $s<1$. Then, by integrating by parts (see, for instance, \cite[Proposition~1]{A15}),
\begin{align*}
    \int_{\Omega}&(-\Delta)^s L_{\Delta} u(x)\phi(x)\ dx
    =\int_{\R^N}L_{\Delta}u(x)(-\Delta)^s\phi(x)\ dx
    =\int_{\Omega} u(x)L_{\Delta}(-\Delta)^s\phi(x)\ dx\\
    &=\int_{\Omega} u(x)(-\Delta)^s L_{\Delta}\phi(x)\ dx
    =\int_{\R^N} (-\Delta)^s u(x) L_{\Delta}\phi(x)\ dx
    =\int_{\Omega}L_{\Delta} (-\Delta)^s u(x) \phi(x)\ dx.
\end{align*}
The claim follows.
\end{proof}

\begin{remark} 
\label{absorbed-term}  
 Note that the term $\int_{\Omega} P_1^{\,c}(x,z) (-\Delta) u (z)\,dz$ in the case $s=1$ in Lemma \ref{lem:interchange} has a counterpart in the case $s<1$ which is completely absorbed into the expression in the RHS of (\ref{eq:lemma-interchange}). Indeed, writing $(-\Delta)^su(x)=U_1(x)+U_2(x)$ with $U_1(x)=1_{\Omega}(x)(-\Delta)^su(x)$ and $U_2(x)=1_{\R^N\setminus \Omega}(x)(-\Delta)^su(x)$ it is easy to check that, for $x\in \Omega$,
\begin{align*}
L_{\Delta}U_2(x)&=\int_{\Omega}P_s^c(x,z)(-\Delta)^su(z)\,dz=-c_N\int_{\R^N\setminus \Omega}\frac{\cN_su(y)}{|x-y|^N}\,dy,
\end{align*}
while 
\begin{align*}
\int_{\Omega} P_1^{\,c}(x,z) (-\Delta) u (z)\,dz=-c_N\int_{\partial \Omega}\frac{\partial_{\nu}u(y)}{|x-y|^N}\,d \sigma(y).
\end{align*}

\end{remark}

\section{Uniform regularity estimates}

\begin{lemma}\label{lemma-uniform-reg-blow-up}
Let $\eps\in(0,\frac{1}{2})$, $a\in(0,\frac{1-2\eps}{N-\eps}),$ $s\in(1-\eps,1)$, and let $g\in L^{\infty}(\Omega;\delta^a)$.  There is $C=C(\Omega,\eps)>0$ such that
\begin{align*}
    \|\cGs_{s} g\|_{C^{s}(\Omega)}<C\qquad \text{ for }s\in (a,1).
\end{align*}
\end{lemma}
\begin{proof}
Let $\eps$, $a$ $s$, and $g$ as in the statement and let $v:=\cGs_{s} g$. In the following, $C>0$ denotes possibly different constants depending at most on $\Omega$ and $\eps$.

Since $g\in L^{\infty}(\Omega;\delta^a)$, using \cite[Theorem 2.4 eq (2.14)]{C99}, we have that
\begin{align*}
|v(x)|
\leq C\int_\Omega G_s(x,y)\delta^{-a}(y)\, dy
\leq C \delta^s(x)\int_\Omega |x-y|^{s-N}\delta^{-a}(y)\, dy.
\end{align*}
Let
\begin{align*}
p=\frac{N-\eps}{N+\eps-1}>1\qquad \text{ and }\qquad q=\frac{N-\eps}{1-2 \eps}>1.
    \end{align*}
By Hölder's inequality,
\begin{align*}
\int_\Omega |x-y|^{s-N}\delta^{-a}(y)\, dy
\leq
\left(\int_\Omega |x-y|^{(s-N)p}\, dy\right)^{\frac{1}{p}}
\left(\int_\Omega \delta^{-aq}(y)\, dy\right)^\frac{1}{q}.
\end{align*}
Then,
\begin{align*}
\int_\Omega |x-y|^{(s-N)p}\, dy
&=\int_{\Omega\cap B_1(x)} |x-y|^{(s-N)p}\, dy+\int_{\Omega\backslash B_1(x)} |x-y|^{(s-N)p}\, dy\\
&\leq |\partial B_1|\int_{B_1(0)} \rho^{(1-\eps-N)p+N-1}\, d\rho+|\Omega|<C,
\end{align*}
because $(1-\eps-N)p+N>0$ since $p<\frac{N}{N+\eps-1}$.  On the other hand, because $a<\frac{1-2\eps}{N-\eps}=\frac{1}{q}$,
\begin{align*}
\int_\Omega \delta^{-aq}(y)\, dy<C,
\end{align*}
see, for instance, \cite[eq. (2.5)]{JSW20}.  As a consequence, \begin{align}\label{v:bdr}
|v(x)|\leq C \delta^s(x)\qquad \text{ for }x\in\Omega.
\end{align}
In particular this implies that $\|v\|_{L^\infty(\Omega)}<C(1+\diam(\Omega))$.

It remains to show that
\begin{align}\label{cla}
    \sup_{\genfrac{}{}{0cm}{}{x,y\in\Omega,}{x\neq y}}\frac{|v(x)-v(y)|}{|x-y|^s} < C.
\end{align}

Now we argue as in \cite[Section A.2]{JSW20}. Let $x,y\in\Omega$, $x\neq y$.   If $|x-y|<\frac{1}{2}\max\{\delta(x),\delta(y)\}$, then either $x\in B_\frac{\delta(y)}{2}(y)$ or $y\in B_\frac{\delta(y)}{2}(x)$.  Let $r:=\frac{\delta(y)}{2}$ and, without loss of generality, assume that $y=0$ and $x\in B_r(y)=B_r$. Then, by \cite[Lemma A.1]{JSW20},
\begin{align*}
    \frac{|v(x)-v(y)|}{|x-y|^s}
    \leq C \delta^s(y)\left(\delta(y)^{-a}
    +\int_{\R^n\backslash B_{r}}\frac{\tau_{N,s}|v(z)|}{|z|^N(|z|^2-r^2)^s}\, dz\right).
\end{align*}
Using \eqref{v:bdr}, a rescaling $\zeta=z/r$, and the fact that $\int_{\R^n\backslash B_1}\frac{\tau_{N,s}(1+|w|)}{|w|^N(|w|^2-1)^s}\, dw<C,$ we obtain
\begin{align*}
\delta^s(y)\int_{\R^N\backslash B_{r}}\frac{\tau_{N,s}|v(z)|}{|z|^N(|z|^2-r^2)^s}\, dz
&\leq C r^s\int_{\Omega\backslash B_{r}}\frac{\tau_{N,s}\delta(z)^s}{|z|^N(|z|^2-r^2)^s}\, dz\\
&=C \int_{(\Omega/r)\backslash B_{1}}\frac{\tau_{N,s}}{|\zeta|^N(|\zeta|^2-1)^s}\left(\frac{\delta(r\zeta)}{r}\right)^s\, d\zeta
\leq C,
\end{align*}
because $\delta(r\zeta)-2r=\delta(r\zeta)-\delta(y)\leq C|r\zeta-y|=Cr|\zeta|$ (recall that $y=0$).

As a consequence, $\frac{|v(x)-v(y)|}{|x-y|^s}<C$ if $|x-y|<\frac{1}{2}\max\{\delta(x),\delta(y)\}$.  On the other hand, if $|x-y|\geq \frac{1}{2}\max\{\delta(x),\delta(y)\}$, then by \eqref{v:bdr},
\begin{align*}
\frac{|v(x)-v(y)|}{|x-y|}
\leq 2^s\frac{|v(x)|}{\delta(x)^s}+2^s\frac{|v(y)|}{\delta(y)^s}\leq C.
\end{align*}
Hence, \eqref{cla} holds and this ends the proof.
\end{proof}

\begin{lemma}\label{lem:continues-with-blow-up}
  Let $\Omega\subset \R^N$ be an open bounded set with $C^2$ boundary, $a\in(0,\frac{2}{4N-1}),$ and let $g\in L^{\infty}(\Omega;\delta^a)$. Then,
  $$
 \cGs_sg \to \cGs_1g\quad\text{ in $C^{\alpha}(\Omega)$ as $s\to 1^-$ for any $\alpha\in(0,1)$.}
  $$
\end{lemma}
\begin{proof}
Since $a<\frac12$, it holds that $g\in L^2(\Omega)$ and thus  $\lim_{s\to 1^-} \cGs_sg=\cGs_1g$ in $L^2(\Omega)$ by \cite[Lemma 5.1]{JSW20}. Now, by Lemma \ref{lemma-uniform-reg-blow-up} (with $\eps=\frac{1}{4}$), it holds that $(\cGs_sg)_s$ is bounded in $C^{\tilde{\alpha}}(\Omega)$ for $s>\tilde{\alpha}>\alpha$ and $\alpha\in(0,1)$. Then the Arzela-Ascoli Theorem implies the convergence of $(\cGs_sg)_s$ in $C^{\alpha}(\Omega)$.
\end{proof}

\subsection*{Declarations}

\noindent{Ethical Approval}: Not applicable.\\ 
{Funding}:  A. Saldaña is supported by UNAM-DGAPA-PAPIIT grant number IA100923 and by CONACYT grants A1-S-10457 and CBF2023-2024-116 (Mexico).\\
{Availability of data and materials}: Not applicable.

\subsection*{Acknowledgments}

A. Saldaña thanks the Goethe University Frankfurt for the kind hospitality.


\end{document}